\documentclass[11pt]{amsart}

\usepackage[T1]{fontenc}
\usepackage{lmodern}
\usepackage{microtype}
\usepackage{mathtools}
\usepackage{xcolor}
\usepackage[margin=1in]{geometry}
\usepackage{enumitem}
\usepackage{needspace}
\usepackage{tikz}
\usepackage{aliascnt}
\usepackage[colorlinks=true,linkcolor=blue!45!black,citecolor=blue!45!black,urlcolor=blue!45!black]{hyperref}
\usepackage[nameinlink,capitalize,noabbrev]{cleveref}

\hypersetup{
  pdftitle={Descent Tops and Pinnacles on 2143- and 3421-Avoiding Permutations},
  pdfauthor={Yue Dong, Lily Li Liu, and Tongyuan Zhao},
  pdfsubject={Permutation patterns, descent tops, pinnacles, and Dumont permutations},
  pdfkeywords={permutation pattern, descent top, pinnacle, Wilf equivalence, Dumont permutation, Cartesian tree, bijection}
}

\numberwithin{equation}{section}
\setlist{leftmargin=2.1em,itemsep=2pt,topsep=4pt}

\newtheorem{theorem}{Theorem}[section]
\newaliascnt{lemma}{theorem}
\newtheorem{lemma}[lemma]{Lemma}
\aliascntresetthe{lemma}
\newaliascnt{proposition}{theorem}
\newtheorem{proposition}[proposition]{Proposition}
\aliascntresetthe{proposition}
\newaliascnt{corollary}{theorem}
\newtheorem{corollary}[corollary]{Corollary}
\aliascntresetthe{corollary}
\theoremstyle{definition}
\newaliascnt{definition}{theorem}
\newtheorem{definition}[definition]{Definition}
\aliascntresetthe{definition}
\theoremstyle{remark}
\newaliascnt{remark}{theorem}

\aliascntresetthe{remark}
\theoremstyle{definition}
\newaliascnt{example}{theorem}
\newtheorem{example}[example]{Example}
\aliascntresetthe{example}

\newcommand{\Av}{\operatorname{Av}}
\newcommand{\DT}{\operatorname{DT}}
\newcommand{\T}{\mathcal{T}}
\newcommand{\dn}{\operatorname{dn}}
\newcommand{\up}{\operatorname{up}}
\newcommand{\id}{\mathrm{id}}
\newcommand{\emptyword}{\varepsilon}
\newcommand{\Pin}{\operatorname{Pin}}
\newcommand{\Pword}{\operatorname{Pword}}
\newcommand{\Val}{\operatorname{Val}}
\newcommand{\Vword}{\operatorname{Vword}}
\newcommand{\Desbot}{\operatorname{Desbot}}
\newcommand{\Asctop}{\operatorname{Asctop}}
\newcommand{\Ascbot}{\operatorname{Ascbot}}
\newcommand{\rev}{\operatorname{rev}}
\newcommand{\inv}{\operatorname{inv}}

\title[Descent tops and pinnacles]{Descent Tops and Pinnacles on
2143- and 3421-Avoiding Permutations}
\author[Y. Dong]{Yue Dong$^{1}$}
\address{$^{1}$School of Mathematical Sciences, Qufu Normal University,
273165 Qufu, P. R. China}
\email{dongyue0810@163.com}
\author[L. L. Liu]{Lily Li Liu$^{1}$}
\address{$^{1}$School of Mathematical Sciences, Qufu Normal University,
273165 Qufu, P. R. China}
\email{liulily@qfnu.edu.cn}
\author[T. Zhao]{Tongyuan Zhao$^{2,*}$}
\address{$^{2}$College of Sciences, China University of Petroleum,
102249 Beijing, P. R. China}
\email{zhaotongyuan@cup.edu.cn}
\thanks{$^{*}$Corresponding author.}
\thanks{Liu is supported by the National Natural Science Foundation of
China (No.~12371330). Zhao was supported by the National Natural Science
Foundation of China (No.~12201641).}
\date{}


\begin{document}

\begin{abstract}
We prove the $2143$--$3421$ case of Burstein's conjecture on the
distribution of descent-top sets by constructing an explicit bijection
between the two avoidance classes. The bijection also preserves the
pinnacle set and reverses the left-to-right order of the pinnacle values.
We combine an involution on maximum Cartesian trees with exchanges of
the final entries of maximal decreasing runs. On an intermediate
avoidance class, we prove that the exchange procedures terminate and
that the output in each direction is independent of the choices made.
This gives an explicit inverse. Restricting the bijection to Dumont
permutations of the first kind proves the conjecture of Burstein and
Jones that $2143$ and $3421$ are Wilf-equivalent on this class.
\end{abstract}

\maketitle

\section{Introduction and main results}\label{sec:introduction}

Write $S_n$ for the permutations of $[n]$. A permutation $\pi\in S_n$
\emph{contains} a pattern $\tau\in S_k$ if some subsequence of $\pi$ is
order-isomorphic to $\tau$; otherwise $\pi$ \emph{avoids} $\tau$. We write
$\Av_n(\tau)$ for the set of $\tau$-avoiding permutations in $S_n$.

For $\pi=\pi_1\cdots\pi_n$, define its set of \emph{descent-top values} by
\[
  \DT(\pi)=\{\pi_i:1\le i<n,\ \pi_i>\pi_{i+1}\}.
\]
Two patterns are Wilf-equivalent if their avoidance classes have the
same cardinality in every length. Refining this equality by a statistic
asks for more: the statistic must have the same distribution on the two
classes. Such refinements were studied by Sagan and
Savage~\cite{SaganSavage2012}; Bloom~\cite{Bloom2014} considered
descent-set refinements for patterns of length four. Here we distinguish
the set of descent-top \emph{values} from the set of descent positions.

Burstein~\cite[Conjecture~11]{Burstein2025} conjectured, in particular,
that $2143$ and $3421$ are Wilf-equivalent on permutations with any fixed
descent-top set. This would imply an earlier conjecture of Burstein and
Jones~\cite[Conjecture~5.1]{BursteinJones2021} on Dumont permutations of
the first kind. Indeed, these are exactly the permutations of $[2n]$
whose descent-top set is $\{2,4,\ldots,2n\}$.
Zhou, Zang, and Yan~\cite{ZhouZangYan2024} proved several other
descent-top-set equivalences from Burstein's conjectures, together with
refinements by peak values. Their results do not cover the pair
$\{2143,3421\}$.

We prove the conjectured descent-top-set equivalence by a bijection that
also controls the peak values and their order. We use the term
\emph{pinnacle} for a peak value: an entry $\pi_i$, with
$2\le i\le n-1$, is a pinnacle if $\pi_{i-1}<\pi_i>\pi_{i+1}$.
Write $\Pin(\pi)$ for the set of these values and
\[
  \Pword(\pi)=\pi_{i_1}\pi_{i_2}\cdots\pi_{i_k},
  \qquad i_1<i_2<\cdots<i_k,
\]
for the word obtained by reading them from left to right.
The term pinnacle follows~\cite{DavisNelsonPetersenTenner2018,Rusu2020,RusuTenner2021};
the same set is called the peak set in~\cite{ZhouZangYan2024}.
Our main result is the following.

\Needspace{10\baselineskip}
\begin{theorem}\label{thm:main}
For every $n\ge0$, there is a bijection
\[
  \Phi:\Av_n(2143)\longrightarrow\Av_n(3421)
\]
such that, for every $\pi\in\Av_n(2143)$,
\begin{equation}\label{eq:main-statistics}
  \DT(\Phi(\pi))=\DT(\pi),
  \qquad
  \Pword(\Phi(\pi))=\rev\bigl(\Pword(\pi)\bigr).
\end{equation}
In particular, $\Phi$ preserves the pinnacle set.
\end{theorem}

The bijection will be written as
\begin{equation}\label{eq:main-bijection}
  \Phi(\sigma)=\dn\bigl(\T(\sigma)\bigr),
  \qquad
  \Phi^{-1}(\rho)=\T\bigl(\up(\rho)\bigr).
\end{equation}
Here $\T$ exchanges the two subtrees at every binary node of the maximum
Cartesian tree, while leaving each unary child on its original side.
The maps $\dn$ and $\up$ repeatedly exchange selected endpoints of
decreasing runs in opposite directions. On a class of words avoiding
the vincular pattern $3\text{-}4\text{-}21$ defined in
\cref{sec:threshold}, we prove that each procedure terminates at a
unique permutation in its connected component of the corresponding
exchange graph. In this class, the
downward terminal permutations avoid $3421$, whereas applying $\T$ to
the upward terminal permutations gives the $2143$-avoiding class.
The tree involution reverses the order of the pinnacles, and the endpoint
exchanges leave that order unchanged.

The ordered refinement in \cref{thm:main} gives a fiberwise statement. For a word $\omega$, write $\omega^{\mathrm r}=\rev(\omega)$.

\begin{corollary}\label{cor:fiberwise}
For every $S\subseteq[n]$ and every word $\omega$ of distinct letters from
$[n]$, the map $\Phi$ restricts to a bijection
\[
\begin{split}
 \{\pi\in\Av_n(2143):\DT(\pi)=S,\ \Pword(\pi)=\omega\}
 \longrightarrow{}&\\
 \{\rho\in\Av_n(3421):\DT(\rho)=S,\ \Pword(\rho)=\omega^{\mathrm r}\}.
\end{split}
\]
Consequently, $2143$ and $3421$ are jointly $(\DT,\Pin)$-Wilf-equivalent.
\end{corollary}

\begin{example}\label{ex:preview}
The bijection sends $15243$ to $34152$ through the intermediate word
$34251$:
\[
  15243\xrightarrow{\T}34251
  \xrightarrow{\text{exchange }2,1}34152.
\]
Both the source and the image have descent-top set $\{4,5\}$, while
their pinnacle words are $54$ and $45$, respectively. We return to this
example after defining the exchanges and give its inverse at the end of
\cref{sec:bijection}.
\end{example}

The proof has three parts. In \cref{sec:preliminaries}, we introduce the
tree involution and the endpoint exchanges, and record the statistics
they preserve. In \cref{sec:characterizations}, we characterize the two
avoidance classes by the absence of exchanges in the corresponding
directions. In \cref{sec:normal-forms}, we prove uniqueness of the two
terminal permutations in each component and assemble the bijection.
\Cref{sec:consequences} gives the symmetry and Dumont consequences.

\section{Trees and endpoint exchanges}\label{sec:preliminaries}

All words below have distinct letters in a fixed finite totally ordered
set. The pattern definitions depend only on relative order, so the
arguments apply to recursive subwords without standardization.

\subsection{Run endpoints and a vincular pattern}\label{sec:threshold}

Let $\pi=\pi_1\cdots\pi_n$ be a word with distinct letters.

\begin{definition}\label{def:endpoint-threshold}
A position $r$ is a \emph{run endpoint} if
\[
  r=n \qquad\text{or}\qquad \pi_r<\pi_{r+1}.
\]
Thus $\pi_r$ is the final entry of a maximal consecutive decreasing run; equivalently, it is not a descent top. Write the run-endpoint positions as
\[
  r_1<r_2<\cdots<r_m,
  \qquad e_i=\pi_{r_i}.
\]
For every position $r$, we use the threshold
\begin{equation}\label{eq:threshold}
 F_\pi(r)=\max\{\pi_p:p<q<r,\ \pi_p<\pi_q\},
\end{equation}
with $F_\pi(r)=-\infty$ if no such pair exists. At a run endpoint, set
\[
  h_i=F_\pi(r_i).
\]
When convenient, we also write $h(e_i)=h_i$.
\end{definition}

The threshold is monotone along the endpoints:
\begin{equation}\label{eq:h-monotone}
  h_1\le h_2\le\cdots\le h_m.
\end{equation}
Whenever $h_i\ne-\infty$, there are positions $p<q<r_i$ such that
\begin{equation}\label{eq:threshold-witness}
  \pi_p=h_i<\pi_q.
\end{equation}

Let
\[
  \nu=3\text{-}4\text{-}21
\]
be the vincular pattern in which the last two entries must be adjacent. Thus $\pi$ contains $\nu$ if there are $p<q<s<n$ such that
\[
  \pi_{s+1}<\pi_s<\pi_p<\pi_q.
\]

\begin{lemma}[Threshold characterization]\label{lem:nu-threshold}
The word $\pi$ contains $\nu$ if and only if there is a descent-top position $s$ such that
\[
  F_\pi(s)>\pi_s.
\]
\end{lemma}

\begin{proof}
If $\pi_p,\pi_q,\pi_s,\pi_{s+1}$ form an occurrence of $\nu$, then $\pi_p<\pi_q$ is an increasing pair before $s$, and $\pi_p>\pi_s$. Hence
\[
  F_\pi(s)\ge \pi_p>\pi_s.
\]
Conversely, suppose that $\pi_s>\pi_{s+1}$ and $F_\pi(s)>\pi_s$. Choose an increasing pair $F_\pi(s)<y$ realizing the threshold. Then
\[
  \pi_{s+1}<\pi_s<F_\pi(s)<y,
\]
so $F_\pi(s),y,\pi_s,\pi_{s+1}$ form an occurrence of $\nu$.
\end{proof}

\subsection{The Cartesian-tree involution}\label{sec:tree}

The \emph{maximum Cartesian tree} of a word $\pi$ is defined recursively
(cf.\ \cite{Vuillemin1980}). If $M$ is the largest letter and
\[
  \pi=A\,M\,B,
\]
then the root is $M$, its left subtree is the maximum Cartesian tree of $A$, and its right subtree is the maximum Cartesian tree of $B$. Inorder traversal recovers $\pi$.

Set $\T(\emptyword)=\emptyword$. For a nonempty word, define $\T$ by exchanging the left and right subtrees at every node having two nonempty subtrees, while preserving the direction of the unique child at a unary node. Equivalently,
\begin{equation}\label{eq:T-recursion}
\T(A M B)=
\begin{cases}
 M\T(B), & A=\emptyword,\\
 \T(A)M, & B=\emptyword,\\
 \T(B)M\T(A), & A,B\ne\emptyword.
\end{cases}
\end{equation}

On $132$-avoiding permutations, this recursive rule is the
descent-top-preserving bijection to $231$-avoiding permutations used by
Burstein~\cite[Theorem~1, equation~(1)]{Burstein2025}. Here we apply the
same rule to arbitrary words with distinct letters. For the example in
\cref{ex:preview}, it gives
\[
  \T(15243)=\T(243)\,5\,\T(1)=34251.
\]
The two binary nodes are labeled $5$ and $4$; exchanging their subtrees
gives the trees in \cref{fig:tree-example}.

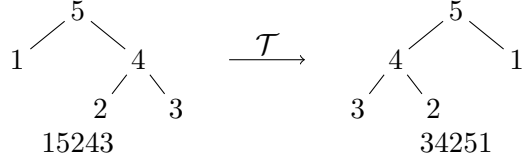
\begin{figure}[htbp]
\centering
\begin{tikzpicture}[x=1cm,y=1cm,every node/.style={inner sep=2pt}]
  \node (a5) at (0,0) {$5$};
  \node (a1) at (-.8,-.65) {$1$};
  \node (a4) at (.8,-.65) {$4$};
  \node (a2) at (.3,-1.3) {$2$};
  \node (a3) at (1.3,-1.3) {$3$};
  \draw (a5)--(a1) (a5)--(a4) (a4)--(a2) (a4)--(a3);
  \node at (0,-1.75) {$15243$};
  \draw[->] (2,-.65)--node[above] {$\T$} (3,-.65);
  \node (b5) at (5,0) {$5$};
  \node (b4) at (4.2,-.65) {$4$};
  \node (b1) at (5.8,-.65) {$1$};
  \node (b3) at (3.7,-1.3) {$3$};
  \node (b2) at (4.7,-1.3) {$2$};
  \draw (b5)--(b4) (b5)--(b1) (b4)--(b3) (b4)--(b2);
  \node at (5,-1.75) {$34251$};
\end{tikzpicture}
\caption{The tree involution in \cref{ex:preview}. Inorder traversal
reads the words beneath the trees.}
\label{fig:tree-example}
\end{figure}

\begin{lemma}\label{lem:dt-right-subtree}
A letter $v$ is a descent top if and only if its Cartesian-tree node has a nonempty right subtree.
\end{lemma}

\begin{proof}
If the node of $v$ has a right subtree, then the entry immediately following $v$ in inorder traversal is the leftmost node of that right subtree. The maximum-heap property makes that entry smaller than $v$, so $v$ is a descent top.

Suppose that the node of $v$ has no right subtree. If $v$ is not the last inorder entry, its inorder successor is the lowest ancestor whose left subtree contains $v$. That ancestor is larger than $v$, so $v$ is followed by a larger entry. If no such ancestor exists, then $v$ is the final entry. In either case, $v$ is not a descent top.
\end{proof}

In particular,
\begin{equation}\label{eq:endpoint-no-right}
  v\text{ is a run endpoint}
  \quad\Longleftrightarrow\quad
  v\text{ has no right subtree}.
\end{equation}

\begin{lemma}\label{lem:T-involution-DT}
The transformation $\T$ is an involution and preserves descent-top values:
\[
  \T^2=\id,
  \qquad
  \DT(\T(\pi))=\DT(\pi).
\]
\end{lemma}

\begin{proof}
A binary node is exchanged twice, while the direction of a unary node is never changed, so $\T^2=\id$. A binary node has a right subtree both before and after the exchange. A unary node keeps its child direction. Thus each individual letter retains the property of having a right subtree, and the result follows from \cref{lem:dt-right-subtree}.
\end{proof}

\begin{lemma}\label{lem:pinnacle-binary}
A letter $v$ is a pinnacle if and only if its node in the maximum Cartesian tree is binary.
\end{lemma}

\begin{proof}
A node has a nonempty left subtree if and only if its entry has an inorder
predecessor that is smaller. If the left subtree is nonempty, the
predecessor is its rightmost node and is smaller by the maximum-heap
property. If the left subtree is empty, then either the node is the first
inorder entry or its predecessor is the lowest ancestor whose right subtree
contains the node; in the latter case the predecessor is larger. By
\cref{lem:dt-right-subtree}, a node has a nonempty right subtree if and only
if its entry is followed by a smaller entry. Hence an entry has two smaller
neighbors precisely when its node has two nonempty subtrees.
\end{proof}

\begin{lemma}[LCA order criterion]\label{lem:lca-order}
Suppose that $u$ occurs to the left of $v$ in $\pi$. Their order is reversed in $\T(\pi)$ if and only if their least common ancestor in the maximum Cartesian tree is binary.
\end{lemma}

\begin{proof}
Let $L$ be the least common ancestor. If $L$ is binary, then either $u$ and $v$ lie on opposite sides of $L$, or one of them equals $L$ and the other lies on one side. Exchanging the two subtrees of $L$ reverses their inorder order.

If $L$ is not binary, then $L$ must be one of $u,v$; otherwise both nodes
would lie in its unique subtree, contradicting the minimality of $L$. Since
the direction of the unique child of $L$ is unchanged, the other node stays
on the same side of $L$ in inorder traversal. Exchanges strictly below $L$
cannot change the order of either node relative to $L$. Thus the order of
$u$ and $v$ is unchanged.
\end{proof}

\begin{proposition}\label{prop:T-pinnacles}
For every word $\pi$ with distinct letters,
\[
  \Pin(\T(\pi))=\Pin(\pi),
  \qquad
  \Pword(\T(\pi))=\rev\bigl(\Pword(\pi)\bigr).
\]
\end{proposition}

\begin{proof}
By \cref{lem:pinnacle-binary}, the pinnacle values are precisely the labels of binary nodes. The transformation $\T$ does not change which nodes are binary, so it preserves the pinnacle set.

Let $u$ and $v$ be two distinct binary nodes. Their least common ancestor is binary: if it equals one of $u,v$, that node is binary by assumption; otherwise $u$ and $v$ lie in different subtrees of their least common ancestor. Hence \cref{lem:lca-order} reverses the order of every pair of pinnacle values. It follows that the entire left-to-right pinnacle word is reversed.
\end{proof}

The least common ancestor of the entries at positions $p<q$ is the node labeled by the maximum entry of the interval $\pi_p\cdots\pi_q$.

\begin{corollary}\label{cor:endpoint-nonbinary-lca}
Let run-endpoint values $a$ and $b$ occur in this order. If their least common ancestor is not binary, then it is the node $b$, and $b$ is the largest entry in the interval from $a$ to $b$.
\end{corollary}

\begin{proof}
A nonbinary least common ancestor must be one of the two nodes. It cannot be $a$: because $b$ lies to the right of $a$, an ancestral node $a$ would have a nonempty right subtree, contradicting \eqref{eq:endpoint-no-right}. Hence it is $b$, which is therefore the interval maximum.
\end{proof}

\subsection{Endpoint exchanges}\label{sec:exchanges}

Consider two run-endpoint positions $r_i<r_j$. Exchange the values at those positions and leave every other position fixed.

\begin{definition}\label{def:admissible}
Such an exchange is \emph{admissible} if it preserves the descent-top value set.
\end{definition}

\begin{lemma}\label{lem:endpoint-positions-preserved}
An admissible exchange preserves all descent positions and all run-endpoint positions.
\end{lemma}

\begin{proof}
The two exchanged values are not descent tops before the exchange. If either became a descent top afterward, the descent-top value set would gain that value. Every unmoved value must also retain its descent-top status, since the values are distinct and the descent-top value set is unchanged. Consequently, every position retains its descent or non-descent status, and hence the set of run-endpoint positions is unchanged.
\end{proof}

\begin{lemma}\label{lem:exchange-pinnacles}
An admissible endpoint exchange preserves the peak positions, the pinnacle values at those positions, and hence the pinnacle word.
\end{lemma}

\begin{proof}
For $2\le i\le n-1$, the position $i$ is a peak position exactly when $i$ is a descent and $i-1$ is not a descent. Thus the peak positions are determined by the descent set.
By \cref{lem:endpoint-positions-preserved}, an admissible exchange preserves every descent position, so it preserves every peak position. The exchanged positions are run endpoints and therefore are not descent positions. In particular, no value occupying a peak position is moved. Thus the pinnacle word is unchanged.
\end{proof}

\begin{lemma}[Local admissibility criterion]\label{lem:local-admissibility}
Let $p<q$ be run-endpoint positions, with
\[
  a=\pi_p,\qquad b=\pi_q.
\]
If $q=p+1$, then exchanging $a$ and $b$ is not admissible. Suppose that $q\ge p+2$.

\begin{enumerate}[label=\textup{(\roman*)}]
\item If $a>b$, the exchange is admissible if and only if
\[
  q=n\quad\text{or}\quad a<\pi_{q+1},
\]
and, whenever the indicated predecessors exist,
\[
  \pi_{p-1}\notin(b,a),
  \qquad
  \pi_{q-1}\notin(b,a).
\]
\item If $a<b$, the exchange is admissible if and only if
\[
  b<\pi_{p+1},
\]
and, whenever the indicated predecessors exist,
\[
  \pi_{p-1}\notin(a,b),
  \qquad
  \pi_{q-1}\notin(a,b).
\]
\end{enumerate}
\end{lemma}

\begin{proof}
If $q=p+1$, then $a<b$ because $p$ is a run endpoint. Since $q$ is also a run endpoint and the letters are distinct, $b$ is not a descent-top value before the exchange. After the exchange, $b$ is immediately followed by $a$, so $b$ becomes a new descent top.

Now assume that $q\ge p+2$. Only the four adjacent comparisons
\[
 (\pi_{p-1},a),\quad (a,\pi_{p+1}),\quad
 (\pi_{q-1},b),\quad (b,\pi_{q+1})
\]
can change, with nonexistent boundary terms omitted.

Suppose first that $a>b$. At position $p$, the value $b$ automatically remains a run endpoint because
\[
  b<a<\pi_{p+1}.
\]
At position $q$, the value $a$ remains a run endpoint exactly when $q=n$ or $a<\pi_{q+1}$. The descent-top status of either predecessor changes exactly when that predecessor lies strictly between $b$ and $a$.

The case $a<b$ is symmetric. The value $a$ automatically remains a run endpoint at $q$, while $b$ remains a run endpoint at $p$ exactly when $b<\pi_{p+1}$. Again, a predecessor changes status exactly when its value lies strictly between $a$ and $b$.
\end{proof}

\begin{corollary}\label{cor:admissible-binary-lca}
The least common ancestor of two admissibly exchangeable run endpoints is binary.
\end{corollary}

\begin{proof}
Otherwise, by \cref{cor:endpoint-nonbinary-lca}, the later endpoint value $b$ is the maximum of the interval between the two endpoints. Moving $b$ to the earlier endpoint position makes it immediately followed by a smaller entry from that interval, so $b$ becomes a descent top. This contradicts admissibility.
\end{proof}

\begin{definition}\label{def:directed-exchanges}
Let $i<j$.
\begin{itemize}
\item We call an admissible exchange a \emph{downward exchange} if
\[
  e_j<e_i<h_i.
\]
\item We call an admissible exchange an \emph{upward exchange} if
\[
  e_i<e_j<h_i.
\]
\end{itemize}
\end{definition}

\section{Pattern avoidance and directed exchanges}\label{sec:characterizations}

We first characterize $3421$-avoidance and then characterize the words
whose images under $\T$ avoid $2143$. These descriptions identify the
terminal permutations needed for the bijection.

\subsection{The downward direction}\label{sec:downward}

\begin{lemma}\label{lem:3421-raw-criterion}
The word $\pi$ contains $3421$ if and only if either it contains $\nu$, or there are endpoint indices $i<j$ such that
\[
  e_j<e_i<h_i.
\]
\end{lemma}

\begin{proof}
Suppose that $e_j<e_i<h_i$. Choose an increasing pair $h_i<y$ realizing $h_i$. The four entries
\[
  h_i,\ y,\ e_i,\ e_j
\]
occur in this order and satisfy
\[
  e_j<e_i<h_i<y,
\]
so they form $3421$. An occurrence of $\nu$ is itself an occurrence of $3421$.

Conversely, consider an occurrence
\[
  x,\ y,\ z,\ w,
  \qquad w<z<x<y.
\]
If $z$ is a descent top, let $d<z$ be its immediate successor. Then $x,y,z,d$ form $\nu$. Hence, if $\pi$ avoids $\nu$, the value $z$ is a run endpoint, say $z=e_i$. Let $e_j$ be the endpoint of the maximal decreasing run containing $w$. Since $z$ occurs before $w$ and is itself an endpoint, this run lies strictly after $z$, so $j>i$. Moreover,
\[
  e_j\le w<z=e_i,
\]
while the increasing pair $x<y$ gives $h_i\ge x>e_i$.
\end{proof}

\begin{lemma}[Minimal-choice lemma]\label{lem:down-choice}
Suppose that $\pi$ avoids $\nu$ and that $e_j<e_i<h_i$ for some $i<j$. Then $\pi$ admits a downward exchange.
\end{lemma}

\begin{proof}
Among all endpoints $e_i$ satisfying
\[
  e_i<h_i
  \quad\text{and}\quad
  \text{some later endpoint is smaller than }e_i,
\]
choose one of minimum value and denote it by $a=e_i$. Among the later endpoints smaller than $a$, choose one of maximum value and denote it by $b=e_j$. Thus
\begin{equation}\label{eq:down-choice-values}
  b<a<h_i.
\end{equation}
Let their positions be $p<q$. Since $a>b$ and $p$ is a run endpoint,
we have $q\ge p+2$. Choose a threshold witness
\[
  H<Y,
  \qquad H=h_i,
\]
with both entries before $p$.

We verify the conditions in \cref{lem:local-admissibility}.

First, suppose that $b<\pi_{p-1}=c<a$. Since $c<a=\pi_p$, position $p-1$ is a run endpoint. Both witness entries $H,Y$ lie before $p$ and are greater than $c$, so neither occupies position $p-1$. Thus they both lie before $c$, giving $h(c)\ge H>c$. The later endpoint $b$ satisfies $b<c$, so $c$ is a smaller eligible choice than $a$, a contradiction.

Second, suppose that $b<\pi_{q-1}=c<a$. Then $c>b$, so $c,b$ are an adjacent descent. Since
\[
  b<c<a<H<Y,
\]
the entries $H,Y,c,b$ form $\nu$, a contradiction.

Finally, suppose that $q<n$ and $b<\pi_{q+1}=s<a$. If $s$ is a descent top, let $d<s$ be its immediate successor. Then $H,Y,s,d$ form $\nu$. If $s$ is a run endpoint, it is a later endpoint strictly between $b$ and $a$, contradicting the maximal choice of $b$.

Thus the local admissibility conditions hold. By \eqref{eq:down-choice-values}, the resulting admissible exchange is a downward exchange.
\end{proof}

\begin{proposition}[Downward characterization]\label{prop:down-characterization}
A word $\pi$ avoids $3421$ if and only if it avoids $\nu$ and admits no downward exchange.
\end{proposition}

\begin{proof}
A downward exchange supplies indices with $e_j<e_i<h_i$, so \cref{lem:3421-raw-criterion} yields an occurrence of $3421$. Conversely, if $\pi$ contains $3421$, then \cref{lem:3421-raw-criterion} gives either an occurrence of $\nu$ or indices $i<j$ with $e_j<e_i<h_i$. In the latter case, if $\pi$ avoids $\nu$, \cref{lem:down-choice} produces a downward exchange.
\end{proof}

\subsection{The upward direction}\label{sec:upward}

We now characterize the words $\pi$ for which $\T(\pi)$ avoids $2143$.

\begin{lemma}\label{lem:nu-to-2143}
If $\pi$ contains $\nu$, then $\T(\pi)$ contains $2143$.
\end{lemma}

\begin{proof}
We use induction on the length. Write $\pi=A M B$, where $M$ is the maximum letter. If an occurrence of $\nu$ lies entirely in $A$ or entirely in $B$, the result follows by induction.

Consider an occurrence crossing $M$. Write its adjacent descent as $z>d$ and its preceding increasing pair as $x<y$, so
\[
  d<z<x<y.
\]
The adjacent descent cannot contain $M$: the alternatives $z=M$ and $d=M$ contradict $z<x$ and $z>d$, respectively. If $z,d$ lie in $A$, then the earlier entries $x,y$ also lie in $A$, contrary to the assumption that the occurrence crosses $M$. Hence $z,d$ lie in $B$, and the selected value $x$ lies in $A$, with
\[
  z<x<M.
\]

The value $z$ is a descent top in $B$. By \cref{lem:T-involution-DT}, it remains a descent top in $\T(B)$; let $d'<z$ be its immediate successor there. Since
\[
  \T(\pi)=\T(B)M\T(A),
\]
the entries
\[
  z,\ d',\ M,\ x
\]
occur in this order and satisfy $d'<z<x<M$. They form $2143$.
\end{proof}

\begin{lemma}\label{lem:up-to-2143}
If $\pi$ admits an upward exchange, then $\T(\pi)$ contains $2143$.
\end{lemma}

\begin{proof}
Let the exchanged endpoint values be
\[
  a=e_i<b=e_j<h_i.
\]
Choose a threshold witness whose lower entry is $x=h_i$, and let $y$ be the
largest entry occurring strictly after this occurrence of $x$ and strictly
before the position of $a$. Such an entry exists and satisfies $y>x$. Thus
the entries occur as
\[
  x,\ y,\ a,\ b,
\]
and their values satisfy
\begin{equation}\label{eq:xyab-values}
  a<b<x<y.
\end{equation}

We show that the least common ancestor of every pair among these four entries is binary. The value $y$ is the maximum of the interval from $x$ to $a$, and it occurs strictly inside that interval. Hence $y$ is binary and is the least common ancestor of each of the pairs $\{x,y\}$, $\{x,a\}$, and $\{y,a\}$. The least common ancestor of $a,b$ is binary by \cref{cor:admissible-binary-lca}. The interval from $x$ to $b$ contains the strictly internal value $y>x,b$, so the least common ancestor of $x,b$ is binary. Finally, if the maximum of the interval from $y$ to $b$ is strictly internal, then the least common ancestor of $y,b$ is binary. If that maximum is $y$, then $y$ has a right subtree containing $b$ and, because it is also the maximum of the interval from $x$ to $a$, a left subtree containing $x$; hence it is again binary.

By \cref{lem:lca-order}, every pairwise order among $x,y,a,b$ is reversed by $\T$. Their order in $\T(\pi)$ is therefore
\[
  b,\ a,\ y,\ x.
\]
Together with \eqref{eq:xyab-values}, this is an occurrence of $2143$.
\end{proof}

The converse will use endpoint pairs satisfying
\begin{equation}\label{eq:up-candidate}
  i<j,\qquad e_i<e_j<h_i,\qquad
  \text{the least common ancestor of $e_i,e_j$ is binary}.
\end{equation}
We first find such a pair and then show that one can be chosen to give
an admissible exchange.

\begin{lemma}\label{lem:subword-lifting}
Let $\pi=A M B$, where $M$ is the maximum letter.
\begin{enumerate}[label=\textup{(\roman*)}]
\item A run endpoint of $A$ or $B$ remains a run endpoint of $\pi$.
\item For an endpoint $v$ in $A$, its threshold is unchanged; for an endpoint $v$ in $B$, its threshold can only increase.
\item The least common ancestor of two nodes both lying in $A$, or both lying in $B$, is unchanged.
\end{enumerate}
Consequently, a pair satisfying~\eqref{eq:up-candidate} in $A$ or $B$
also satisfies it in $\pi$.
\end{lemma}

\begin{proof}
All successor relations internal to $A$ and $B$ are unchanged. The last entry of $A$, if present, is followed by the larger value $M$, and hence remains a run endpoint. This proves (i). The prefix preceding an endpoint in $A$ is unchanged, while an endpoint in $B$ only gains additional preceding entries, proving (ii). The Cartesian trees of $A$ and $B$ occur unchanged as the two subtrees of the root $M$, proving (iii). The final assertion follows immediately.
\end{proof}

\Needspace{8\baselineskip}
\begin{lemma}[$132$ extraction]\label{lem:132-extraction}
Suppose that $\gamma$ avoids $\nu$ and that $\T(\gamma)$ contains entries
\[
  w,\ y,\ x
\]
in this order, with
\[
  w<x<y.
\]
Then $w$ is a run endpoint in $\gamma$, and its threshold satisfies
\[
  h(w)\ge x.
\]
\end{lemma}

\begin{proof}
We argue by induction on the length. Write $\gamma=C N D$, where $N$ is the maximum letter. If all three selected entries lie in $\T(C)$ or all lie in $\T(D)$, the result follows from the induction hypothesis and \cref{lem:subword-lifting}.

Suppose that the occurrence crosses the root. Then $C,D\ne\emptyword$ and
\[
  \T(\gamma)=\T(D)N\T(C).
\]
The smallest selected entry $w$ cannot equal $N$, and $x$ cannot equal $N$ because $x<y$. Since $w$ is the first selected entry and $x$ is the last, crossing the root forces
\[
  w\in D,
  \qquad x\in C.
\]
In the original word, the increasing pair $x<N$ occurs before $w$.

If $w$ were a descent top, let $d<w$ be its immediate successor. Then
\[
  d<w<x<N,
\]
and $x,N,w,d$ would form $\nu$, a contradiction. Thus $w$ is a run endpoint. The increasing pair $x<N$ lies before it, so $h(w)\ge x$.
\end{proof}

\begin{lemma}\label{lem:low-value-inversion}
Suppose that every descent-top value of a word $\gamma$ is greater than $c$. If the subsequence of $\T(\gamma)$ induced by the entries smaller than $c$ contains an inversion, then $\gamma$ has two run-endpoint values $a$ and $b$ such that
\[
  a<b<c,
\]
with $a$ occurring before $b$, and their least common ancestor is binary.
\end{lemma}

\begin{proof}
Every entry smaller than $c$ is a run endpoint. Choose an inversion $z\cdots w$ in the indicated subsequence of $\T(\gamma)$, so that $w<z<c$.

If $z$ also occurred before $w$ in $\gamma$, then their least common ancestor could not be nonbinary: by \cref{cor:endpoint-nonbinary-lca}, the later value $w$ would have to be the interval maximum, whereas the interval contains $z>w$. Hence their least common ancestor would be binary, and \cref{lem:lca-order} would place $w$ before $z$ in $\T(\gamma)$, a contradiction.

Therefore $w$ occurs before $z$ in $\gamma$. Their order is reversed by $\T$, so \cref{lem:lca-order} gives a binary least common ancestor. Taking $a=w$ and $b=z$ proves the claim.
\end{proof}

\begin{lemma}[Extraction of an endpoint pair]\label{lem:extract-proto-up}
Suppose that $\pi$ avoids $\nu$ but $\T(\pi)$ contains $2143$. Then $\pi$ contains an endpoint pair satisfying~\eqref{eq:up-candidate}.
\end{lemma}

\begin{proof}
We use induction on the length. Write $\pi=A M B$, and choose an occurrence
\[
  z,\ w,\ y,\ x
\]
of $2143$ in $\T(\pi)$, listed in their order of appearance. Thus
\begin{equation}\label{eq:2143-values}
  w<z<x<y.
\end{equation}
If the occurrence lies entirely in $\T(A)$ or entirely in $\T(B)$, the result follows from induction and \cref{lem:subword-lifting}.

Assume that the occurrence crosses the root. Then $A,B\ne\emptyword$ and
\[
  \T(\pi)=\T(B)M\T(A).
\]
The first selected value $z$ must lie in $\T(B)$; it cannot equal $M$ because $z<y$. The values $w$ and $x$ cannot equal $M$ because $w$ is the smallest selected value and $x<y$. There are two cases.

\smallskip
\noindent\emph{Case 1: only $z$ lies in $\T(B)$.}
Then $w,y,x$ lie in $\T(A)$ and satisfy $w<x<y$. By \cref{lem:132-extraction}, the value $w$ is a run endpoint in $A$ and
\[
  h_A(w)\ge x.
\]
In the original word, the increasing pair $x<M$ lies before $z$. If $z$ were a descent top, with immediate successor $d<z$, then $x,M,z,d$ would form $\nu$. Hence $z$ is also a run endpoint. By \eqref{eq:2143-values},
\[
  w<z<x\le h(w),
\]
and $w,z$ lie in opposite subtrees of the binary root $M$. Thus $(w,z)$ is an endpoint pair satisfying~\eqref{eq:up-candidate}.

\smallskip
\noindent\emph{Case 2: both $z$ and $w$ lie in $\T(B)$.}
Because the occurrence crosses the root, its final entry $x$ lies in $\T(A)$. The increasing pair $x<M$ occurs before all entries of $B$ in the original word. If $B$ had a descent top $t<x$, with immediate successor $d<t$, then $x,M,t,d$ would form $\nu$. Hence every descent-top value in $B$ is greater than $x$.

The entries $z,w$ form an inversion in the subsequence of $\T(B)$ induced by values smaller than $x$. By \cref{lem:low-value-inversion}, the word $B$ has run endpoints
\[
  a<b<x
\]
in this order, with a binary least common ancestor. The increasing pair $x<M$ lies before $a$, so $h(a)\ge x>b$. Thus $(a,b)$ is an endpoint pair satisfying~\eqref{eq:up-candidate}.

These cases exhaust all occurrences crossing the root.
\end{proof}

\begin{lemma}[Admissible refinement]\label{lem:proto-to-up}
If $\pi$ avoids $\nu$ and contains an endpoint pair satisfying~\eqref{eq:up-candidate}, then $\pi$ admits an upward exchange.
\end{lemma}

\begin{proof}
Among all endpoint pairs satisfying~\eqref{eq:up-candidate}, first choose one whose second endpoint has minimum endpoint index $j$. Subject to this choice, choose the first endpoint with maximum value. Denote the selected pair by
\[
  a=e_i<b=e_j<h_i,
\]
and let its positions be $p<q$. Since its least common ancestor is binary, the two endpoint positions are not adjacent: adjacent endpoints satisfy $a<b$, and their least common ancestor is the later endpoint $b$, which has no right subtree.

Choose a threshold witness
\[
  H<Y,
  \qquad H=h_i,
\]
with both entries before $p$. We verify the three possible obstructions in \cref{lem:local-admissibility}.

Suppose first that $a<\pi_{p-1}=c<b$. Then $c>a$, so $c,a$ are an adjacent descent. Since $c<b<H<Y$, the witness entries $H,Y$ occur strictly before $c$, and $H,Y,c,a$ form $\nu$, a contradiction.

Next suppose that $a<\pi_{p+1}=s<b$. If $s$ is a descent top, let $d<s$ be its immediate successor. Then $H,Y,s,d$ form $\nu$. If $s$ is a run endpoint, then
\[
  a<s<b<h(s),
\]
because $h(s)\ge H$. We claim that $s,b$ have a binary least common ancestor. Otherwise, by \cref{cor:endpoint-nonbinary-lca}, the value $b$ is the maximum of the interval from $s$ to $b$. Since $a$ is immediately to the left of $s$ and $a<s<b$, the value $b$ is also the maximum of the interval from $a$ to $b$. This would make the least common ancestor of $a,b$ the nonbinary endpoint $b$, contradicting condition~\eqref{eq:up-candidate}. Thus $(s,b)$ is an endpoint pair satisfying~\eqref{eq:up-candidate} with the same second endpoint and a larger first value, contradicting the choice of $a$.

Finally suppose that $a<\pi_{q-1}=c<b$. Then $c<b$, so $c$ is a run endpoint. If $a,c$ had a nonbinary least common ancestor, then \cref{cor:endpoint-nonbinary-lca} would make $c$ the maximum of the interval from $a$ to $c$. Since $b$ immediately follows $c$ and $b>c$, the value $b$ would then be the maximum of the interval from $a$ to $b$, forcing the least common ancestor of $a,b$ to be the nonbinary endpoint $b$. This is impossible. Hence $(a,c)$ is an endpoint pair satisfying~\eqref{eq:up-candidate} whose second endpoint precedes $b$, contradicting the minimal choice of $j$.

Therefore
\[
  b<\pi_{p+1},
  \qquad
  \pi_{p-1},\pi_{q-1}\notin(a,b),
\]
with nonexistent predecessors omitted. The smaller value $a$ automatically remains a run endpoint after it is moved to position $q$, because $b$ was an endpoint. By \cref{lem:local-admissibility}, the exchange is admissible. Since $a<b<h_i$, it is an upward exchange.
\end{proof}

\begin{proposition}[Upward characterization]\label{prop:up-characterization}
The word $\T(\pi)$ avoids $2143$ if and only if $\pi$ avoids $\nu$ and admits no upward exchange.
\end{proposition}

\begin{proof}
If $\pi$ contains $\nu$, then \cref{lem:nu-to-2143} gives an occurrence of $2143$ in $\T(\pi)$. If $\pi$ admits an upward exchange, the same conclusion follows from \cref{lem:up-to-2143}.

Conversely, suppose that $\pi$ avoids $\nu$ but $\T(\pi)$ contains $2143$. By \cref{lem:extract-proto-up}, the word $\pi$ contains an endpoint pair satisfying~\eqref{eq:up-candidate}, and by \cref{lem:proto-to-up}, it admits an upward exchange.
\end{proof}

\section{Normal forms and the bijection}\label{sec:normal-forms}

\subsection{The exchange graph and its invariants}

Fix a finite set of letters. Form a graph on the words using these
letters once each, joining two words precisely when an upward exchange or a
downward exchange transforms one into the other. Only exchanges satisfying
\cref{def:directed-exchanges} are edges of this graph. We work with its
$\nu$-avoiding components; the next lemma shows that avoidance of $\nu$
is preserved along every edge.

\begin{lemma}[Threshold invariance]\label{lem:threshold-invariance}
An upward exchange or downward exchange preserves $F_\pi(r)$ for every position $r$. Consequently, it preserves the endpoint threshold vector $(h_1,\ldots,h_m)$ and the property of avoiding $\nu$.
\end{lemma}

\begin{proof}
Suppose that the exchange occurs at endpoint positions $p=r_i<q=r_j$, exchanging the values $a,b$, and set $H=h_i$. In either direction,
\[
  a<H,
  \qquad b<H,
\]
and an increasing pair with lower entry $H$ lies entirely before $p$.

For $r\le p$, neither exchanged position occurs in the prefix defining $F_\pi(r)$, so the threshold is unchanged. Let $r>p$. The fixed witness with lower entry $H$ remains before $r$, so the threshold is at least $H$ both before and after the exchange. Every increasing pair created or destroyed by the exchange contains $a$ or $b$. If one of these values is the lower entry, that lower entry is less than $H$; if it is the upper entry, the lower entry is smaller still. Thus every changed candidate for the maximum in \eqref{eq:threshold} is less than $H$, while a candidate equal to $H$ remains present. Therefore $F_\pi(r)$ is unchanged.

The endpoint positions are preserved by \cref{lem:endpoint-positions-preserved}, so the vector $(h_1,\ldots,h_m)$ is unchanged. Finally, \cref{lem:nu-threshold} characterizes $\nu$ using the descent-top positions, their values, and the corresponding thresholds. All of these are preserved.
\end{proof}

\begin{definition}\label{def:movable-slot}
We call an endpoint slot $i$ \emph{movable} if
\[
  e_i<h_i.
\]
\end{definition}

\begin{lemma}\label{lem:exchange-invariants}
Upward exchanges and downward exchanges are inverse operations. In each component, the following data are fixed:
\begin{enumerate}[label=\textup{(\roman*)}]
\item the run-endpoint positions, the values at all other positions, and the threshold vector;
\item the set of movable endpoint slots;
\item the values in the nonmovable endpoint slots;
\item the set of values occupying the movable endpoint slots.
\end{enumerate}
\end{lemma}

\begin{proof}
Suppose that an upward exchange at slots $i<j$ exchanges
\[
  a<b<h_i.
\]
After the exchange, the earlier slot contains $b$ and the later slot contains $a$. By \cref{lem:threshold-invariance}, the threshold vector is unchanged, so the reverse exchange satisfies
\[
  a<b<h_i,
\]
which is exactly the downward exchange condition. Admissibility is symmetric. The reverse implication is identical.

Because $h_i\le h_j$, both exchanged values are smaller than the thresholds at both exchanged slots. Thus both slots are movable before and after the exchange, and all other slots are unchanged. The listed invariants follow.
\end{proof}

An upward exchange strictly increases the inversion number of the endpoint-value sequence, while a downward exchange strictly decreases it. Indeed, exchanging $a<b$ at endpoint indices $i<j$ increases that inversion number by
\[
  1+2\#\{k:i<k<j,\ a<e_k<b\}>0.
\]
Thus every sequence of exchanges in a fixed direction terminates.

\subsection{Uniqueness of the two normal forms}

\begin{lemma}[Unique downward normal form]\label{lem:unique-down}
Every $\nu$-avoiding component contains a unique word with no downward exchange.
\end{lemma}

\begin{proof}
If the sequence of values in the movable slots has an inversion $e_i>e_j$ with $i<j$, then
\[
  e_j<e_i<h_i,
\]
and \cref{lem:down-choice} produces a downward exchange. Conversely, every downward exchange displays such an inversion. Hence a word has no downward exchange exactly when the values in its movable slots increase from left to right.

Within a fixed component, the movable slots and their value set are fixed by \cref{lem:exchange-invariants}. There is only one way to place those distinct values increasingly in those slots. Since every downward path terminates, this arrangement occurs and is the unique downward terminal word.
\end{proof}

For a $\nu$-avoiding word $\pi$, denote its unique downward normal form
by $\dn(\pi)$.

\begin{definition}\label{def:movable-block}
We call a maximal interval of consecutive positions, each of which is a movable endpoint position, a \emph{movable block}.
\end{definition}

The values in a movable block are strictly increasing, since each of its positions is a run endpoint.

\begin{lemma}\label{lem:cross-block}
Let $\pi$ avoid $\nu$. Suppose that endpoints $a=e_i$ and $b=e_j$ belong to different movable blocks and satisfy
\[
  i<j,
  \qquad a<b<h_i.
\]
Then $\pi$ admits an upward exchange.
\end{lemma}

\begin{proof}
Let the positions of $a,b$ be $p<q$. We first show that there is a descent top strictly between them. If not, the entire interval from $p$ to $q$ is strictly increasing and every position in it is a run endpoint. Since $a,b$ lie in different movable blocks, an intermediate endpoint value $f$ is nonmovable. Strict increase gives
\[
  a<f<b.
\]
Nonmovability gives $f\ge h(f)$, while monotonicity of the threshold gives $h(f)\ge h_i>b$, a contradiction.

Let $t$ be a descent top strictly between $a$ and $b$. Choose a threshold witness
\[
  H<Y,
  \qquad H=h_i,
\]
before $a$. If $t<H$, and $d<t$ is the immediate successor of $t$, then $H,Y,t,d$ form $\nu$. Since all values are distinct, $t\ne H$, and hence
\[
  t>H>b>a.
\]
Therefore the maximum of the interval from $a$ to $b$ lies strictly inside that interval, so their least common ancestor is binary. The pair $(a,b)$ is an endpoint pair satisfying~\eqref{eq:up-candidate}, and \cref{lem:proto-to-up} yields an upward exchange.
\end{proof}

\begin{lemma}[Unique upward normal form]\label{lem:unique-up}
Every $\nu$-avoiding component contains a unique word with no upward exchange.
\end{lemma}

\begin{proof}
Every upward path terminates because each exchange strictly increases
the inversion number of the endpoint-value sequence. Thus a terminal
word exists in every component.

Suppose that distinct words $P,Q$ in the same component admit no
upward exchange. The movable blocks occupy the same positions in both
words. Let $C$ be the first such block on which their values differ, and
write its endpoint indices as $j_1<\cdots<j_s$. Put
\[
  A_r=e^P_{j_r},\qquad B_r=e^Q_{j_r}\qquad(1\le r\le s).
\]
Both sequences are strictly increasing. Let $k$ be their first differing
index. Interchanging $P,Q$ if necessary, assume $A_k<B_k$.
The sets of values in $C$ differ, since an increasing arrangement of a
fixed set is unique. Let $t\ge k$ be the first index for which
$B_t\notin\{A_1,\ldots,A_s\}$.

We claim that $B_t>A_t$. This is immediate if $t=k$. If $t>k$, the
$t-k$ values $B_k,\ldots,B_{t-1}$ belong to the $A$-set and are all
greater than $A_k$. Only $t-k-1$ members of that set lie strictly
between $A_k$ and $A_t$. Hence $B_{t-1}\ge A_t$, and therefore
$B_t>B_{t-1}\ge A_t$.

Set $v=B_t$. In $P$, the value $v$ does not occur in $C$ and cannot
occur before $C$: all earlier movable blocks agree in $P,Q$, and the
values at every other position are fixed throughout the component.
Otherwise $Q$ would contain $v$ twice. Thus $v$ occurs in a later
movable block of $P$. In $Q$, it occupies the movable slot $j_t$, so
$v<h_{j_t}$. The two endpoints in $P$ therefore satisfy
\[
  A_t<v<h_{j_t}
\]
and belong to different movable blocks. By \cref{lem:cross-block},
$P$ admits an upward exchange, a contradiction. The upward terminal word
is therefore unique.
\end{proof}

For a $\nu$-avoiding word $\pi$, denote its unique upward normal form
by $\up(\pi)$.

\begin{corollary}\label{cor:canonical-transversals}
In the graph of $\nu$-avoiding words whose edges are upward exchanges or
downward exchanges, every connected component contains exactly one $3421$-avoiding word and exactly one word $p$ for which $\T(p)$ avoids $2143$. These are respectively the downward and upward normal forms of the component.
\end{corollary}

\begin{proof}
By \cref{prop:down-characterization}, the $3421$-avoiding words are precisely the $\nu$-avoiding words with no downward exchange, and uniqueness follows from \cref{lem:unique-down}. By \cref{prop:up-characterization}, the words $p$ for which $\T(p)$ avoids $2143$ are precisely the $\nu$-avoiding words with no upward exchange, and uniqueness follows from \cref{lem:unique-up}.
\end{proof}

The two normal forms can be computed by making any available exchange
in the chosen direction until none remains. If there are $m$ run
endpoints, this takes at most $\binom{m}{2}$ exchanges, since the endpoint
inversion number changes strictly at each step. The preceding lemmas
show that the output is independent of these choices.

\Needspace{14\baselineskip}
\begin{example}\label{ex:exchange-choices}
Consider the word $3\mid52\mid64\mid71$. Its endpoint positions are
$(1,3,5,7)$, with endpoint values $(3,2,4,1)$ and thresholds
$(-\infty,3,5,6)$. Thus the last three endpoint slots are movable.
There are two possible first downward exchanges. The complete downward
exchange graph of its component consists of the two paths
\[
\begin{aligned}
3526471&\xrightarrow{\text{exchange }2,1}3516472
         \xrightarrow{\text{exchange }4,2}3516274,\\
3526471&\xrightarrow{\text{exchange }4,1}3526174
         \xrightarrow{\text{exchange }2,1}3516274.
\end{aligned}
\]
Both choices give the same downward normal form $3516274$, whose
movable endpoint values are $1,2,4$ in increasing order. Reversing the
arrows gives the upward exchanges and the common upward normal form
$3526471$. Throughout the component, the descent-top set is
$\{5,6,7\}$ and the pinnacle word is $567$.
\end{example}

\subsection{Proof of the main theorem}\label{sec:bijection}

\begin{proof}[Proof of \cref{thm:main}]
Define $\Phi=\dn\circ\T$ and $\Psi=\T\circ\up$. We verify their
domains, show that they are inverse, and then check the statistics.

Let $\sigma\in\Av_n(2143)$ and put $p=\T(\sigma)$. Since $\T^2=\id$, the word $\T(p)=\sigma$ avoids $2143$. By \cref{prop:up-characterization}, the word $p$ avoids $\nu$ and has no upward exchange. Thus it is the unique upward normal form of its component.

Set $q=\dn(p)$. By \cref{lem:threshold-invariance}, the word $q$ still avoids $\nu$, and by definition it has no downward exchange. Hence \cref{prop:down-characterization} gives $q\in\Av_n(3421)$. Thus $\Phi$ is well-defined.

Conversely, let $\rho\in\Av_n(3421)$. By \cref{prop:down-characterization}, it avoids $\nu$ and is the unique downward normal form of its component. Put $p=\up(\rho)$. Then $p$ is the unique upward normal form in that component, so \cref{prop:up-characterization} gives
\[
  \T(p)\in\Av_n(2143).
\]

Uniqueness of the two normal forms yields
\[
  \up(\dn(p))=p,
  \qquad
  \dn(\up(\rho))=\rho.
\]
Together with $\T^2=\id$, these identities imply
\[
  \Psi(\Phi(\sigma))=\sigma,
  \qquad
  \Phi(\Psi(\rho))=\rho.
\]
Finally, $\T$ preserves descent-top values by \cref{lem:T-involution-DT}, and every endpoint exchange is admissible by definition. Therefore $\Phi$ preserves $\DT$. By \cref{prop:T-pinnacles}, the map $\T$ reverses the pinnacle word, while \cref{lem:exchange-pinnacles} shows that every exchange used by $\dn$ preserves it. Hence
\[
  \Pword(\Phi(\sigma))=\rev\bigl(\Pword(\sigma)\bigr).
\]
\end{proof}

\begin{example}[The inverse in \cref{ex:preview}]\label{ex:inverse}
The three maximal decreasing runs of $34251$ are
$3\mid42\mid51$. Their endpoint values are $(3,2,1)$ and their
thresholds are $(-\infty,3,4)$. The only downward exchange swaps $2$
and $1$, giving $3\mid41\mid52=34152$. Its inverse is the upward
exchange of these same endpoint values. Thus the full construction and
its inverse are
\[
\begin{aligned}
  15243&\xrightarrow{\T}34251\xrightarrow{\dn}34152,\\
  34152&\xrightarrow{\up}34251\xrightarrow{\T}15243.
\end{aligned}
\]
The descent-top set is $\{4,5\}$ at every stage. The pinnacle words are
$54,45,45$ along the first line, making the roles of the two operations
visible separately.
\end{example}

\section{Consequences}\label{sec:consequences}

\subsection{Reversal and complementation}

For a permutation $\pi=\pi_1\cdots\pi_n$, define
\[
\begin{aligned}
 \Desbot(\pi)&=\{\pi_{i+1}:\pi_i>\pi_{i+1}\},\\
 \Asctop(\pi)&=\{\pi_{i+1}:\pi_i<\pi_{i+1}\},\\
 \Ascbot(\pi)&=\{\pi_i:\pi_i<\pi_{i+1}\}.
\end{aligned}
\]
An entry $\pi_i$, with $2\le i\le n-1$, is a \emph{vale} if
$\pi_{i-1}>\pi_i<\pi_{i+1}$; write $\Val(\pi)$ for the vale set and
$\Vword(\pi)$ for the vale values in left-to-right order.

Let $r$ and $c$ denote reversal and complementation, and put $rc=r\circ c=c\circ r$. For a set or word $A$ on $[n]$, write $c(A)$ for entrywise complementation $a\mapsto n+1-a$. In particular, complementation does not reverse the order of a word.

\begin{proposition}\label{prop:symmetric-companions}
The following conjugates of $\Phi$ are bijections with the indicated refinements:
\[
\begin{array}{c|c|c}
\text{bijection} & \text{pattern classes} & \text{preserved set-valued statistics}\\ \hline
r\Phi r & \Av_n(3412)\longrightarrow\Av_n(1243) & (\Asctop,\Pin)\\
c\Phi c & \Av_n(3412)\longrightarrow\Av_n(2134) & (\Ascbot,\Val)\\
rc\Phi rc & \Av_n(2143)\longrightarrow\Av_n(4312) & (\Desbot,\Val)
\end{array}
\]
The first map reverses the pinnacle word, while the second and third reverse the vale word.
\end{proposition}

\begin{proof}
The pattern identities are
\[
\begin{aligned}
 r(2143)&=3412, & r(3421)&=1243,\\
 c(2143)&=3412, & c(3421)&=2134,\\
 rc(2143)&=2143, & rc(3421)&=4312.
\end{aligned}
\]
The relevant statistic identities are
\[
\begin{aligned}
 \DT(r\pi)&=\Asctop(\pi),
 &\Pword(r\pi)&=\rev(\Pword(\pi)),\\
 \DT(c\pi)&=c(\Ascbot(\pi)),
 &\Pword(c\pi)&=c(\Vword(\pi)),\\
 \DT(rc\pi)&=c(\Desbot(\pi)),
 &\Pword(rc\pi)&=c(\rev(\Vword(\pi))).
\end{aligned}
\]
Conjugating \eqref{eq:main-statistics} by the corresponding symmetry gives all three assertions; entrywise complementation is injective, so the complemented set equalities are equivalent to the stated ones.
\end{proof}

\subsection{Generating functions}

For indeterminates $\mathbf{x}=(x_1,\ldots,x_n)$ and $\mathbf{y}=(y_1,\ldots,y_n)$, define
\[
  F_n^\tau(\mathbf{x},\mathbf{y})
  =\sum_{\pi\in\Av_n(\tau)}
    \prod_{a\in\DT(\pi)}x_a
    \prod_{b\in\Pin(\pi)}y_b.
\]

By \cref{thm:main},
\[
  F_n^{2143}(\mathbf{x},\mathbf{y})
  =F_n^{3421}(\mathbf{x},\mathbf{y})\qquad(n\ge0).
\]
In particular, the two classes have the same joint distribution of
the number of descents and the number of pinnacles.

For a word $w=w_1\cdots w_k$ of distinct letters, let $\inv(w)$ be its inversion number. For fixed $S,P\subseteq[n]$, define
\[
  G_{n,S,P}^{\tau}(q)
  =\sum_{\substack{\pi\in\Av_n(\tau)\\
                    \DT(\pi)=S,\ \Pin(\pi)=P}}
     q^{\inv(\Pword(\pi))}.
\]

For $k=|P|$, \cref{cor:fiberwise} gives
\[
  G_{n,S,P}^{2143}(q)
  =q^{\binom{k}{2}}G_{n,S,P}^{3421}(q^{-1}),
\]
because reversing a word of $k$ distinct letters changes its inversion
number from $a$ to $\binom{k}{2}-a$.

\subsection{Dumont permutations of the first kind}\label{sec:dumont}

A permutation $\pi\in S_{2n}$ is a \emph{Dumont permutation of the first
kind} if every even entry is followed by a smaller entry, while every odd
entry is followed by a larger entry or is the final entry
\cite{Dumont1974}. Equivalently, its descent-top set is
$\{2,4,\ldots,2n\}$. Let $D^1_{2n}(\tau)$ denote the subset avoiding
$\tau$. Thus \cref{cor:fiberwise} immediately gives the following
refinement of \cite[Conjecture~5.1]{BursteinJones2021}.

\begin{corollary}\label{cor:dumont}
For every $n\ge0$ and every word $\omega$ of distinct letters from $[2n]$,
\[
\begin{split}
 &\bigl|\{\pi\in D^1_{2n}(2143):\Pword(\pi)=\omega\}\bigr|\\
 &\hspace{25mm}=
 \bigl|\{\rho\in D^1_{2n}(3421):\Pword(\rho)=\rev(\omega)\}\bigr|.
\end{split}
\]
Consequently, the pinnacle set has the same distribution on these two
classes, and in particular
\[
  \bigl|D^1_{2n}(2143)\bigr|
  =
  \bigl|D^1_{2n}(3421)\bigr|.
\]
\end{corollary}

For example, $563421\in D^1_6(2143)$ has descent-top set $\{2,4,6\}$
and pinnacle word $64$. Its tree image is $214365$, which already avoids
$3421$, so
\[
  \Phi(563421)=\dn\bigl(\T(563421)\bigr)=214365.
\]
The image has the same descent-top set and pinnacle word $46$.


\begin{thebibliography}{99}

\bibitem{Bloom2014}
J.~Bloom,
\newblock A refinement of Wilf-equivalence for patterns of length $4$,
\newblock \emph{J. Combin. Theory Ser. A} \textbf{124} (2014), 166--177,
\href{https://doi.org/10.1016/j.jcta.2014.01.001}{doi:10.1016/j.jcta.2014.01.001}.

\bibitem{Burstein2025}
A.~Burstein,
\newblock Distribution of sets of descent tops and descent bottoms on restricted permutations,
\newblock \emph{Discrete Math. Theor. Comput. Sci.} \textbf{26} (2025), no.~1, Paper No.~8,
\href{https://doi.org/10.46298/dmtcs.12636}{doi:10.46298/dmtcs.12636}.

\bibitem{BursteinJones2021}
A.~Burstein and O.~Jones,
\newblock Enumeration of Dumont permutations avoiding certain four-letter patterns,
\newblock \emph{Discrete Math. Theor. Comput. Sci.} \textbf{22} (2021), no.~2, Paper No.~7,
\href{https://doi.org/10.46298/dmtcs.6174}{doi:10.46298/dmtcs.6174}.

\bibitem{DavisNelsonPetersenTenner2018}
R.~Davis, S.~A. Nelson, T.~K. Petersen, and B.~E. Tenner,
\newblock The pinnacle set of a permutation,
\newblock \emph{Discrete Math.} \textbf{341} (2018), no.~11, 3249--3270,
\href{https://doi.org/10.1016/j.disc.2018.08.011}{doi:10.1016/j.disc.2018.08.011}.

\bibitem{Dumont1974}
D.~Dumont,
\newblock Interpr\'etations combinatoires des nombres de Genocchi,
\newblock \emph{Duke Math. J.} \textbf{41} (1974), no.~2, 305--318,
\href{https://doi.org/10.1215/S0012-7094-74-04134-9}{doi:10.1215/S0012-7094-74-04134-9}.

\bibitem{Rusu2020}
I.~Rusu,
\newblock Sorting permutations with fixed pinnacle set,
\newblock \emph{Electron. J. Combin.} \textbf{27} (2020), no.~3, Paper No.~P3.23,
\href{https://doi.org/10.37236/9231}{doi:10.37236/9231}.

\bibitem{RusuTenner2021}
I.~Rusu and B.~E. Tenner,
\newblock Admissible pinnacle orderings,
\newblock \emph{Graphs Combin.} \textbf{37} (2021), no.~4, 1205--1214,
\href{https://doi.org/10.1007/s00373-021-02306-9}{doi:10.1007/s00373-021-02306-9}.

\bibitem{SaganSavage2012}
B.~E. Sagan and C.~D. Savage,
\newblock Mahonian pairs,
\newblock \emph{J. Combin. Theory Ser. A} \textbf{119} (2012), no.~3, 526--545,
\href{https://doi.org/10.1016/j.jcta.2011.11.003}{doi:10.1016/j.jcta.2011.11.003}.

\bibitem{Vuillemin1980}
J.~Vuillemin,
\newblock A unifying look at data structures,
\newblock \emph{Commun. ACM} \textbf{23} (1980), no.~4, 229--239,
\href{https://doi.org/10.1145/358841.358852}{doi:10.1145/358841.\allowbreak358852}.

\bibitem{ZhouZangYan2024}
R.~D.~P. Zhou, Y.~Zang, and S.~H.~F. Yan,
\newblock Further refinements of Wilf-equivalence for patterns of length $4$,
\newblock \emph{J. Combin. Theory Ser. A} \textbf{204} (2024), Paper No.~105863,
\href{https://doi.org/10.1016/j.jcta.2024.105863}{doi:10.1016/j.jcta.2024.105863}.

\end{thebibliography}
\end{document}